\documentclass[11pt]{amsart}
\usepackage{ae,amsfonts,euscript,enumerate}
\usepackage{amsmath,euscript}
\usepackage{amssymb}
\usepackage{amsthm}
\usepackage{enumerate}

\usepackage{amsfonts}
\usepackage{amssymb,amsmath,array}
\usepackage{slashbox}
\usepackage[usenames,dvipsnames]{color}
\newtheorem{thm}{Theorem}[section]
\newtheorem{lem}[thm]{Lemma}

\newtheorem{prop}[thm]{Proposition}

\theoremstyle{definition}

\newtheorem{quest}{Question}

\newcommand{\st}{\,|\,}

\def\co{{\mathcal O}}

\def\oqmm13{\co_q(M_{1,3})}
\def\oqm23{\co_q(M_{2,3})}

\usepackage{amssymb}
\usepackage{amsmath}
\usepackage[latin1]{inputenc}

\usepackage{amsmath,euscript}
\usepackage{amssymb}
\usepackage{amsthm}
\usepackage{enumerate}
\usepackage{amsfonts}
\usepackage{colortbl}
\usepackage{a4wide}
\usepackage{amssymb,amsmath,array}

\title[]{Embeddings of quotient division algebras of rings of differential operators}

\author{Jason P.~Bell}
\thanks{The authors thank NSERC for its generous support.}


\address{Jason Bell and Ritvik Ramkumar\\
Department of Pure Mathematics\\
University of Waterloo\\
Waterloo, ON, N2L 3G1\\
Canada}

\email{jpbell@uwaterloo.ca}
\email{rramkuma@uwaterloo.ca}
\author{Colin Ingalls}

\address{Colin Ingalls\\
Department of Mathematics\\
University of New Brunswick\\
Fredricton, NB, \\
Canada}
\email{cingalls@math.unb.ca}

\author{Ritvik Ramkumar}

\begin{document}
\bibliographystyle{plain}
%%%%%%%%%%%%%%%%%%%%%%%%%%%%%%%%%%%%%%%%%%%%%%%%%%%%%%%%%%%%%%%%%%%%%%%%%%%%%%%%%%%%%%%%%%%%%%%%%%%%%%%%%%%%%%%%%%%%%%

\begin{abstract} Let $k$ be an algebraically closed field of characteristic zero, 
let $X$ and $Y$ be smooth irreducible algebraic curves over $k$, and let $D(X)$ and $D(Y)$ denote respectively the quotient division rings of the ring of differential operators of $X$ and $Y$.  We show that if there is a $k$-algebra embedding of $D(X)$ into $D(Y)$ then the genus of $X$ must be less than or equal to the genus of $Y$, answering a question of the first-named author and Smoktunowicz.  
\end{abstract}
\maketitle
\section{Introduction}  
One of the central results in algebraic geometry is the birational classification of surfaces.  In particular, this work shows that the most crucial birational invariants of a surface are the Kodaira dimension, the geometric genus, the irregularity, and the plurigenera.    A related question concerns when there exists a dominant rational map from a smooth irreducible projective surface $X$ to another smooth surface $Y$.  Algebraically, this corresponds to the existence of an embedding of the field of rational functions of $Y$ into the corresponding field for $X$. Many of these birational invariants give restrictions upon when such a map can exist.  For example, if $X$ is a smooth irreducible surface over a base field of characteristic zero and there is dominant rational map to another smooth surface $Y$ then all of the above listed birational invariants of $X$ must be at least those of $Y$ (see Lemma~\ref{bndschar0} for details).  In the case of curves, the Riemann-Hurwitz theorem \cite[IV, Corollary 2.4]{Hart} gives that if $X$ and $Y$ are smooth irreducible projective curves and there is surjective morphism from $X$ to $Y$ then the genus of $X$ must be at least that of $Y$.  

One has noncommutative analogues of these questions, which we now describe.  If $A$ is a graded noetherian domain of GK dimension $3$ that is generated in degree one, and we let $\mathcal{C}$ denote the category of finitely generated graded right $A$-modules
modulo the subcategory of torsion modules, then we can create a noncommutative analogue of ${\rm Proj}(A)$ by taking the triple $(\mathcal{C},\mathcal{O}, s)$,
where $\mathcal{O}$ is the image of the right module $A$ in $\mathcal{C}$ and $s$ is the autoequivalence of $\mathcal{C}$ defined by the shift operator on graded modules, as described in \cite{Art}.  Furthermore, $A$ has a graded quotient division ring, which we denote $Q_{\rm gr}(A)$, which is formed
by inverting the nonzero homogeneous elements of $A$. Then the graded form of Goldie's theorem \cite{GS} gives that there is a division ring $D$ and
automorphism $\sigma$ of $D$ such that
$$Q_{\rm gr}(A) = D[t, t^{-1}; \sigma].$$
We can then view $D$ as being the ``function field'' of $X = {\rm Proj}(A)$ and in the commutative setting this construction coincides exactly with the ordinary field of rational functions on an irreducible projective variety.  We can then view ${\rm Proj}(A)$ and ${\rm Proj}(B)$ as being birationally isomorphic if their corresponding division rings are isomorphic.  

Unlike in the commutative case, however, much less is known in the noncommutative setting and in general there is a dearth of noncommutative birational invariants outside of various analogues for transcendence degree \cite{GK, YZ, Z, Z2}.  Despite the lack of noncommutative invariants, there are large classes of well-studied examples of noncommutative surfaces and Artin \cite{Art} has given a proposed birational classification of graded $k$-algebra domains $A$ of GK dimension 3 (possessing additional homological properties which we do not mention here).   In this case, we have $Q_{\rm gr}(A)=D[x,x^{-1};\sigma]$ for some division ring $D$ and, in analogy with the birational classification of surfaces mentioned earlier, one would like to understand the type of division rings that can occur.  In particular, Artin claims that under the hypotheses on $A$ from his paper, the division ring $D$ must satisfy at least one of the following properties:
\begin{enumerate}
\item $D$ is finite-dimensional over its centre, which is a finitely generated extension of $k$
of transcendence degree $2$;
\item $D$ is isomorphic to a Skylanin division ring;
\item $D$ is isomorphic to the quotient division ring of $K[t; \sigma]$ where $K$ is a finitely generated field extension of $k$ of genus $0$ or $1$ and $\sigma$ is a $k$-algebra automorphism of $K$;
\item $D$ is isomorphic to the quotient division ring of $K[t; \delta]$ where $K$ is a finitely generated field extension of $k$ and $\delta$ is a $k$-linear derivation of $K$.
\end{enumerate}
We point out that the first class encompasses all of the fields of rational functions on algebraic surfaces and has lumped them together as part of a single class.  In this sense, the classification is much less specific than in the commutative setting, but from the noncommutative point of view one can see the first case as being essentially ``understood''.  (Work of Chan and the second-named author \cite{CI} provides a birational classification in this case.  In particular, it shows that such division rings are either ``ruled'', ``del Pezzo'', or have a minimal model which is unique up to Morita equivalence.)  We also note that division rings of types (2), (3), and (4) can be finite over their centres, and so when we talk about these latter types we specifically mean only the division rings of this type that are infinite-dimensional over their centres.  Rings of the form $K[t; \delta]$ with $K$ a finitely generated field extension of $k$ of transcendence degree one and $\delta$ a derivation of $K$ are birationally isomorphic to rings of
differential operators of a smooth curve.  The most important ring of this type is the first Weyl algebra, which has generators $x$ and $y$ and the relation $xy-yx=1$.  

Just as one would like to understand conditions which must be met to have a dominant rational map of surfaces in terms of invariants, one would also like to understand when one can have an embedding of division rings between two division rings on Artin's list.  In fact, in some cases it is not hard to see that embeddings cannot occur except for trivial reasons (see \S 3 for more details).  The most difficult cases appear to be when one is dealing with division rings obtained from rings of differential operators on curves (with the exception of the Weyl algebra) and with the Sklyanin algebras.  In the former case, the first-named author along with Smoktunowicz \cite[Conjecture 4.1]{BSm} conjectured that a type of noncommutative version of the Riemann-Hurwitz theorem might hold, in the sense of showing that an embedding can occur only if an inequality holds between the genera of the curves.  Our main result is to prove this conjecture.  We note that we will speak of the genus of a field of rational functions $F$ of a smooth curve, and by this we simply mean the geometric genus of a smooth irreducible projective curve  that has $F$ as its field of rational functions. 
\begin{thm} Let $k$ be an algebraically closed field of characteristic zero, let $F$ and $K$ be finitely generated field extensions of $k$ of transcendence degree one, and let $\mu$ and $\delta$ be nonzero derivations of $F$ and $K$ respectively.  If $F(t;\mu)$ embeds in $K(x;\delta)$ then the genus of $F$ is less than or equal to the genus of $K$.
\label{thm: main}
\end{thm}
We note that if either $\mu$ or $\delta$ is the zero derivation then the theorem holds trivially (see the remarks in \S 3 for more details concerning the case when exactly one is zero; when both are zero this follows from the irregularity inequality mentioned earlier).

The proof of Theorem \ref{thm: main} is complicated by the fact that noncommutative localization is notoriously poorly behaved.  For example, rings of the form $K[t;\delta]$ can be viewed as noncommutative deformations of the polynomial ring $K[t]$ and have an associated graded ring of this form.  This fact often simplifies the study of their ring theoretic properties.  On the other hand, for $K$ of characteristic zero, when one forms the quotient division ring by inverting nonzero elements, the resulting division ring is unwieldy and contains a copy of the free algebra on two generators unless $\delta$ is zero (this can be seen from the fact that the division ring $K(t;\delta)$ then contains a copy of the Weyl algebra and one can then use a result of Makar-Limanov \cite{ML}).  Thus the problem of understanding embeddings of $F[t;\mu]$ into $K[x;\mu]$ is considerably simpler than the problem of understanding embeddings of their quotient division rings.  Our approach in proving Theorem \ref{thm: main} is to use reduction modulo primes.  Here one again must be careful about what we even mean.  It is known that rings of the form $F[t;\mu]$ satisfy a polynomial identity when $F$ is a a field of positive characteristic and thus one can hope that commutative methods might apply.  The one subtlety, however, is that while one can expect to obtain information about $F[t;\mu]$ by ``reducing modulo primes'' one cannot in general expect to gain understanding about its quotient division ring in this manner.  The reason for this is that if one has a finitely generated noncommutative $\mathbb{Z}$-algebra then reducing mod a prime $p$ and then localizing is generally straightforward, but trying to reduce mod $p$ after localizing will in many cases yield the zero ring.

The strategy to get around these difficulties is to suppose that we have an embedding $F(t;\mu)$ into $K(x;\delta)$ for fields of rational functions $F$ and $K$ of curves $X$ and $Y$.  Then this embedding gives an embedding of $C[t;\mu]$ into $K(x;\delta)$ where $C$ is a finitely generated $\mathbb{Z}$-algebra, closed under the derivation $\mu$, such that $C$ contains a set of generators for $F$ as a field extension of $k$.  Although $C[t;\mu]$ is mapping into $K(x;\delta)$, we in fact show that it maps into well-behaved localization of a ring of the form $A[x;\delta]$ where $A$ is again a finitely generated $\mathbb{Z}$-algebra.  In particular, we show that we are still able to reduce modulo primes in ${\rm Spec}(A)$ in our localization.  We now reduce modulo various prime ideals and then localize to obtain embeddings of division rings that are finite over their centres and we also show that these embeddings yield embeddings of the respective centres.  We then use properties of flat morphisms of schemes to show that the centres of these division rings are isomorphic to function fields of ruled surfaces and we finally use results from algebraic geometry to show that the embedding of the centres can occur only if we have the desired inequality for the genera of $X$ and $Y$.  

The outline of this paper is as follows.  In \S 2, we prove Theorem \ref{thm: main} and then in \S 3, we give general remarks about the embedding question for other division rings on Artin's list.

\section{Proof of Theorem \ref{thm: main}}

We begin with a lemma, which we will use to construct a suitable model for our ring of differential operators that will allow us to reduce mod $p$.  
\begin{lem} Let $k$ be an algebraically closed field of characteristic zero, let $F$ be a finitely generated extension of $k$ of transcendence degree one and genus $g$ equipped with a nonzero $k$-linear derivation $\delta$, and let $B\subseteq F$ be a finitely generated $\mathbb{Z}$-algebra.  Then there exist finitely generated $\mathbb{Z}$-subalgebras $C_0$ and $C$ of $k$ and $F$ respectively such that the following properties hold:
\begin{enumerate}
\item $\delta(C)\subseteq C$;
\item there exists $s\in C$ such that $\delta(s)$ and $s$ are units in $C$;
\item $C\supseteq B$;
%\item $C_0=C\cap \colinchange{\sout{\mathbb{C}} k}$;
\item each maximal ideal $P$ of $C_0$ has the property that $PC$ is a prime ideal of $C$ and
$C/PC \otimes_{C_0/P} \overline{C_0/P}$ is an integral domain whose field of fractions is a field extension $F_P$ of $\overline{C_0/P}$ of transcendence degree one and genus $g$;
\item there is some fixed $d\ge 1$ such that for a Zariski dense set of maximal ideals $P$ of $C_0$, the field $F_P$ is an extension of the $\overline{C_0/P}$-subfield of $F_P$ generated by the image of $s$ of fixed degree 
$$d \ = \ [F_P: \overline{C_0/P}(s)].$$
\end{enumerate}\label{lem:1.1}
\end{lem} 
\begin{proof}
We can view $F$ as the field of rational functions on some smooth irreducible projective curve $X$.  We fix an embedding $f: X\to \mathbb{P}_k^m$, and we may assume without loss of generality that $m$ is minimal.  We have that $f(X)$ is the zero set of a homogeneous ideal $I$ in $k[x_0,\ldots ,x_m]$.  We choose a finite set of generators for $I$ and we let $S$ denote the set of coefficients of the polynomials in our finite generating set.  Then the homogeneous coordinates induce rational functions $$f_{i,j}=\left(x_i/x_j\right)\big|_{X}$$ for $i\neq j$ and by minimality of $m$, the $f_{i,j}$ can be extended to regular and non-constant maps on a dense neighbourhood of $X$.  Since $\delta$ is nonzero there is some $s$ among the $f_{i,j}$ such that $\delta(s)\neq 0$.  Then for $k\neq \ell$ we have $\delta(f_{k,\ell})=P_{k,\ell}/Q_{k,\ell}$ for some polynomials $P_{k,\ell}$ and $Q_{k,\ell}$ in the $f_{i,j}$.  We also note that the inclusion of fields $k(s)\to F$ induces a non-constant morphism $s : X\to \mathbb{P}^1$ and we let $d$ denote the degree of this map.

We then take $A$ to be a finitely generated $\mathbb{Z}$-algebra that contains the elements of $S$, the coefficients of the $P_{k,\ell}$ and $Q_{k,\ell}$, and the inverses of the nonzero coefficients of the $Q_{k,\ell}$.  Then we obtain a model $\mathcal{X}$ for $X$ over ${\rm Spec}(A)$, by viewing our generating set as lying in $A[x_0,\ldots ,x_m]$ and taking the zero set of the ideal in $A[x_0,\ldots ,x_m]$ generated by these elements.  We then have a morphism $\tilde{s} : \mathcal{X}\to \mathbb{P}^1_{A}$, and a morphism $\chi : \mathbb{P}^1_{A}\to  {\rm Spec}(A)$, and we let $\phi=\chi\circ \tilde{s} : \mathcal{X}\to {\rm Spec}(A)$.  By generic flatness \cite[Th\'eor\`eme 6.9.1]{EGA IV_2} there is some nonzero $f\in A$ such that if we replace $A$ by $A_f$ and replace $\mathcal{X}$ accordingly,
the morphisms $\phi$ and $\tilde{s}$ are flat.  Then for each prime ideal $t\in {\rm Spec}(A)$, 
we have a fibre, which we consider as a closed subscheme $\mathcal{X}_t$ of $\mathbb{P}^m_{{\rm Frac}(A/t)}$, and we have that $\mathcal{X}_t$ has constant arithmetic genus \cite[III, Cor. 9.10]{Hart}.  Since the generic fibre $\mathcal{X}_{\eta}$ has the property that if we extend the base field we obtain $X$, and since $X$ is smooth and irreducible, we see that the arithmetic genus of $\mathcal{X}_{\eta}$ is equal to the geometric genus of $X$ and so the arithmetic genus of each $\mathcal{X}_t$ is equal to $g$.  

Now the set of $t$ in ${\rm Spec}(A)$ for which the fibre $\mathcal{X}_t$ is geometrically irreducible is constructible  \cite[p. 36, Th\'eoreme 4.10]{Jou}; that is, it is a finite union of sets of the form $U\cap V$ where $U$ is open and $V$ is closed.   Similarly, since the generic fibre is smooth, the fibres $\mathcal{X}_t$ are smooth on a Zariski open set of ${\rm Spec}(A)$ \cite[Proposition 17.7.11(ii)]{EGA IV_4}. Since the generic fibre, $\mathcal{X}_{\eta}$, is also geometrically irreducible, we then see that there is some nonzero $h\in A$ such that for $t\in {\rm Spec}(A_h)\subseteq {\rm Spec}(A)$, the fibres are geometrically irreducible and smooth.  Since the arithmetic genera of $\mathcal{X}_t$ are all equal to $g$ for $t$ in the open subset ${\rm Spec}(A_h)$, by smoothness we then have that the geometric genus of each $\mathcal{X}_t$ is $g$ for $t\in {\rm Spec}(A_h)$.  We now replace $A$ by $A_h$ and again adjust our model accordingly.  We pick an affine open subset $U$ of $\mathcal{X}$ such that the finite set of generators for $B$, $s$ and $s^{-1}$, $\delta(s)$ and $\delta(s)^{-1}$ and the functions $f_{i,j}$ are regular on $U$.  Now the collection of regular functions on $U$ is a finitely generated $A_h$ algebra and if $u_1,\ldots ,u_k$ are generators then there is some $b$ that is regular on $U$ such that $\delta(u_i) \in \mathcal{O}_{\mathcal{X}}(U)[1/b]$.  Since derivations behave well with respect to localization, it is then easy to check that $\mathcal{O}_{\mathcal{X}}(U)[1/b]$ is closed under application of $\delta$.  In particular, we can refine $U$ if necessary and assume that $\delta$ preserves the ring of functions that are regular on $U$.  Then if we let $C_0=A_h$ and $C=\mathcal{O}_{\mathcal{X}}(U)$ then $C\subseteq F$ is a $C_0$-algebra and by construction for every maximal ideal $P$ of $C_0$ we have that $CP$ is a prime ideal of $C$ since the corresponding fibre $\mathcal{X}_P$ is geometrically irreducible (in the scheme theoretic sense); and $C/CP$ is $C_0/P$-algebra of Krull dimension one.  Moreover, by geometric irreducibility of the fibres, we have $C/CP\otimes_{C_0/P} \overline{C_0/P}$ is an integral domain and its field of fractions is a transcendence degree one extension of $\overline{C_0/P}$ of genus $g$, by the remarks above.  We also note that our construction gives that $C\supseteq B$ and $s,s^{-1}\in C$ and $\delta(C)\subseteq C$.  It only remains to show (5).

Let $\mathcal{O}(1)$ denote a very ample invertible sheaf on $\mathbb{P}^1_A$.  Then since $\tilde{s}$ is a finite map, we obtain an invertible sheaf $\mathcal{L} :=\tilde{s}^*\mathcal{O}(1)$ on $\mathcal{X}$, which is ample over $A$.   This sheaf $\mathcal{L}$ then gives us a second embedding of $\mathcal{X}$ into some $\mathbb{P}^m_A$.  By generic flatness, we again have that on an open subset of ${\rm Spec}(A)$ the degree (for this embedding induced by $\mathcal{L}$) of $\mathcal{X}_t$ is constant \cite[III, Cor. 9.10]{Hart} and this must be equal to the degree of the generic fibre, $\mathcal{X}_{\eta}$, which (again by extending the base and obtaining $X$) is the limit as $n$ tends to infinity of $h^0(\tilde{s}^*(\mathcal{O}(n)))/n ={\rm deg}(\tilde{s})=d$.  This says that the degree of the field extension of ${\rm Frac}(C/P)$ is $d$-dimensional over the extension generated by the image of $s$ and this gives (5) and completes the proof.
\end{proof}
Once we reduce modulo $p$, our ring of differential operators will satisfy a polynomial identity.  The embedding questions will then reduce to embedding questions for division rings that are finite-dimensional over their centres.  Our first result shows that under certain circumstances, an embedding of division rings will induce an embedding of the centres.  Given a ring $R$, we let $Z(R)$ denote its centre.
 \begin{lem} Let $D_1$ and $D_2$ be two division rings both containing a central subfield $k$ and suppose that $[D_1:Z(D_1)]=[D_2:Z(D_2)]$.  If there is an injective $k$-algebra homomorphism $\phi: D_1\to D_2$, then $\phi(Z(D_1))\subseteq Z(D_2)$.  
 \label{lem:Z2}
\end{lem}
\begin{proof} Suppose that this does not hold and let $K\supset Z(D_2)$ denote the subalgebra of $D_2$ generated by $Z(D_2)$ and $\phi(Z(D_1))$.  We note that $K$ is commutative since $Z(D_2)$ commutes with everything in $\phi(Z(D_1))$ and $\phi(Z(D_1))$ is commutative.  Moreover, since $D_2$ is finite-dimensional over its centre, we see that $K$ is an integral domain that is finite-dimensional over $Z(D_2)$ and hence must be a field.  By assumption, $K\neq Z(D_2)$.  Notice that $K$ centralizes $E:=\phi(D_1)\subseteq D_2$.  Let $L$ be the division subring of $D_2$ generated by $E$ and $K$.  Then $[L:K]\le [D_2:K]<[D_2:Z(D_2)]$.  It follows that the dimension of $L$ over its centre, which contains $K$, is less than the order of $D_2$ over its centre.  In particular, the PI degree of $L$ is strictly less than the PI degree of $D_2$.  But $D_1\cong E$ embeds in $L$ and so the PI degree of $D_1$ is strictly less than the PI degree of $D_2$, a contradiction.   The result follows.
\end{proof}
To apply the preceding lemma, we need a characterization of the centre of a division ring obtained from the skew polynomial rings over the function field of a curve in positive characteristic.   The following two lemmas give this characterization.  We recall that $K^{\langle p\rangle}=\{ a^p \st a \in K \}$.

 \begin{lem} Let $p$ be a prime number and suppose that $K$ is a finitely generated transcendence degree one extension of an algebraic extension $E$ of $\mathbb{F}_p$ and that $K$ is degree $<p$ over a purely transcendental extension of $E$.   Then $[K:K^{\langle p\rangle}]=p$.
  \label{lem: Kp}
 \end{lem}
 \begin{proof} By assumption, we have that $K$ is a degree $d$ extension of $E(s)$ for some $s\in K$ and some $d<p$.  Thus $K$ is separable and so $K$ is generated by two elements $s,t\in K$ by the primitive element theorem, and there is some irreducible polynomial $P(x,y)$ of degree $d$ in $y$ with $P(s,t)=0$.  Let $F=E(s,t^p)$.  Then
 $[K:F]\in \{1,p\}$ since $t\in K$ has annihilating polynomial $x^p-t^p$ over $F$, and this is either irreducible over $F$ or is the $p$-th power of a linear polynomial over $F$.  But
 $[K:F]\le [K:E(s)]\le {\rm deg}_y(P)<p$ and so we see that $K=F$.   Then we have $K=E(s,t^p)$ and since $K^{\langle p\rangle} \supseteq E(s^p, t^p)$, we see that $[K:K^{\langle p\rangle}]\le [E(s):E(s^p)]=p$.  Thus it suffices to show that 
 $[K:K^{\langle p\rangle}]\ge p$.  Since $K$ is finitely generated and not algebraic over $E$, $K$ is not perfect and so we have that there is some $u\in K$ that is not in $K^{\langle p\rangle }$.  Then $u$ is a root of the polynomial $x^p-u^p \in K^{\langle p\rangle}[x]$ and this is irreducible since $u\not \in K^{\langle p\rangle}$.  In particular, $[K:K^{\langle p\rangle}] \ge p$ and the result follows.
 \end{proof}

\begin{lem} Let $p$ be a prime number and suppose that $K$ is a finitely generated transcendence degree one extension of an algebraic extension $E$ of $\mathbb{F}_p$ and that $K$ is degree $<p$ over a purely transcendental extension of $E$.  If $\delta$ is a nonzero $E$-linear derivation of $K$ then $Z(K(x;\delta))=K^{\langle p\rangle}(x^p)\cong K(t)$. \label{lem:Z}
\end{lem}
\begin{proof} Pick $s\in K$ such that $[K:E(s)]<p$.  Then $\delta(s)=\beta$ for some $\beta\in K$.  Notice that the $E$-linear derivation $\mu$ of $E(s)$ given by differentiation with respect to $s$ extends uniquely to a derivation of $K$ since $K$ is a finite separable extension of $E(s)$.  In particular, we see that $\delta=\beta \mu$ and $K(x;\delta)=K(x;\mu)$.  Thus we may assume without loss of generality that $\delta=\mu$ and $\delta(s)=1$.  Let $Z=Z(K(x;\delta))$.  For $\alpha\in K$ we have $\delta(\alpha^p)=0$ and so $K^{\langle p\rangle} \subseteq Z$.  Also for $\alpha\in K$, since $K$ has characteristic $p$, we have $[x^p,\alpha]={\rm ad}_x^p(\alpha)=\delta^p(\alpha)$.  Notice that $\delta^p={\rm ad}_{x^p}$ is a derivation of $K$ and it annihilates $E(s)$.  Since $K$ is a separable extension of $E(s)$, we see that $\delta^p(K)=0$ and so $x^p\in Z$.  Now we claim that $[Z(x): Z]=p$.  To see this, observe that $x$ satisfies the polynomial equation $t^p-x^p=0$ in $Z[t]$ and this is irreducible unless $x\in Z$.  It follows that $[Z(x):Z]\in \{1,p\}$.  Since $\delta$ is nonzero, we see that $x\not\in Z$ and so we obtain the claim.  
 
 We have $[K: K^{\langle p\rangle}]=p$ by Lemma \ref{lem: Kp} and so $[K(x:\delta): K^{\langle p\rangle}(x^p)]=p^2$.  Thus $[K(x:\delta): Z]\le p^2$.  But whenever $F$ is a maximal subfield of $K(x;\delta)$ we have $[K(x;\delta):Z]= [F:Z]^2$.  In particular, if we pick a maximal subfield containing $Z(x)$, we see that $[F:Z]^2\ge [Z(x):Z]^2=p^2$.  Thus $[K(x:\delta): Z]= p^2$ and so $Z=K^{\langle p\rangle}(x^p)$.  To get the final isomorphism, notice that the map
 $f(t)\mapsto f(t)^p$ gives an isomorphism from $K(t)$ to $K^{\langle p\rangle}(t^p)$.
  \end{proof}
We next prove a few results that are well-known, but for which we are unaware of proper references.  We first prove Lemma~\ref{bndschar0}, which gives many of the claimed inequalities on birational invariants for $X$ and $Y$ when there is a dominant rational map from $Y$ to $X$.  We then give a non-embedding result (Lemma \ref{lem:genus}) that can apply to centres.  We point out that Lemma \ref{lem:genus} immediately follows form Lemma \ref{bndschar0} in the separable case, but we require a more general version.

 \begin{lem} \label{bndschar0}  Let $k$ be a base field and suppose that 
$F \subseteq K$ are finitely generated fields over $k$ and that the extension
$F \subseteq K$ is separable.  Suppose that the transcendence degrees of $F,K$
are two or the characteristic of $k$ is zero. 
Let $X,Y$ be smooth models for $F,K$ respectively.
Then $h^0(X,\Omega^j_X) \leq h^0(Y,\Omega^j_Y)$ and $h^0(X,\omega^{\otimes n}_X) \leq h^0(X,\omega^{\otimes n}_X)$
for $n \geq 0$, and the Kodaira dimension of $Y$ is at least that of $X$.
\end{lem}
\begin{proof}
We have a dominant rational map $Y \dashrightarrow X$.  Due to our assumptions on characteristic and dimension, we may resolve indeterminacies of the map to obtain a regular map $\pi:\widetilde{Y} \rightarrow X$ where $\widetilde{Y}$ is smooth.
Since the above numbers are all birational invariants we may replace $Y$ with $\widetilde{Y}.$
Now since the extension $F \subseteq K$ is separable, $\pi$ is generically \'etale so $\pi^* \Omega^j_X$ is isomorphic to $\Omega^j_{{Y}}$ generically.  Since $\pi^* \Omega^j_X,\Omega^j_{\widetilde{Y}}$ 
are locally free, the natural map $\pi^* \Omega^j_X \rightarrow \Omega^j_{\widetilde{Y}}$ is injective.  Now applying the global section functor and the projection formula yields the result.   The second inequality follows by the same argument, and since the Kodaira dimension is the growth of the plurigenera, we are done.\end{proof}

\begin{lem} Let $k$ be an algebraically closed field and let $F$ and $K$ be function fields of smooth projective irreducible curves over $k$ of genera $g_F$ and $g_K$ respectively.  If there is a $k$-algebra embedding of $F(t)$ into $K(t)$ then $g_F\le g_K$. \label{lem:genus}
\end{lem}
\begin{proof}
Choose a smooth minimal model $X$ for $F(t)$.  Note that $\rho:X \rightarrow C$ is ruled over a curve $C$ with $k(C) = F$.  If $g_F=0$ we are done, so let us suppose that $g_F>0$.
We can choose a smooth model $Y$ of $K(t)$ and we will have a dominant rational map $Y \dashrightarrow X$.  By resolving the singularities of the map we may replace $Y$ with a smooth model where we have a regular map $\pi:Y \rightarrow X$.
Now $Y$ is birational to a surface ruled over a curve $D$ with $k(D)=K$,
so we have a map $Y \rightarrow D$ with fibres that are trees of rational curves.  By Tsen's Theorem~\cite{Tsen}, we have a section $s:D 
\rightarrow Y$.  Let us now consider the map $\psi = \rho \circ \pi \circ s : D \rightarrow X$.
If the image of $\psi$ is $C$ we are done, otherwise the image $\psi$ must be a point $p$ in $C$.  So we see that $D$ maps to the fibre $\rho^{-1}(p) =F \simeq \mathbb{P}^1$.  Now consider a fibre $F'$ in $Y$ over a point $q \in D$.  Now $\pi(F')$ must be connected and since all the components of $F'$ are rational curves, 
the map $\rho \circ \pi$ must be constant on $F'$.  Since $F'$ meets $D$ we see that the image of $F'$ in $X$ is contained in $F$.  So the map $Y \rightarrow X$ is not dominant.
\end{proof}

We are now ready to prove our main result.
\begin{proof}[Proof of Theorem \ref{thm: main}]  Let $g_F$ and $g_K$ denote the genera of $F$ and $K$ respectively.

By Lemma \ref{lem:1.1}, there exist finitely generated $\mathbb{Z}$-algebras $C_0$ and $C$ such that $C[t;\mu]$ is a subring of $F(t;\mu)$, $C$ is a finitely generated $C_0$-algebra, and $C$ and $C_0$ satisfy the following conditions:
\begin{enumerate}
\item $\delta(C)\subseteq C$;
\item there exists $s\in C$ such that $\mu(s)=1$ and $s$ is a unit in $C$;
%\item $C_0=C\cap \colinchange{\sout{\mathbb{C}} k}$;
\item each maximal ideal $P$ of $C_0$ has the property that $PC$ is a prime ideal of $C$ and
$C/PC \otimes_{C_0/P} \overline{C_0/P}$ is an integral domain whose field of fractions is a field extension $F_P$ of $\overline{C_0/P}$ of transcendence degree one and genus $g_F$;
\item there is some fixed $d\ge 1$ such that for a Zariski dense set of maximal ideals $P$ of $C_0$, the field $F_P$ is an extension of the $\overline{C_0/P}$-subfield of $F_P$ generated by the image of $s$ of fixed degree $d$. 
\end{enumerate}

Then the embedding of $F(t;\mu)$ into $K(x;\delta)$ gives an embedding $\iota$ of $C[t;\mu]$.  Let $c_1,\ldots ,c_r$ be generators for $C$ as a $\mathbb{Z}$-algebra.  Then 
there exist elements $p_i(x),q_i(x)\in F[x,\delta]$ such that $\iota(c_i)=p_i(x)q_i(x)^{-1}$ and there exist elements $r(x),s(x)\in F[x,\delta]$ such that $\iota(t)=r(x)s(x)^{-1}$.

Let $B$ denote the $\mathbb{Z}$-algebra generated by the coefficients of $p_i(x),q_i(x),r(x),s(x)$ as well as the inverses of all coefficients of leading monomials.  Then by Lemma \ref{lem:1.1}, there exists a finitely generated $\mathbb{Z}$-subalgebra $A_0$ of $k$ and a finitely generated $A_0$-algebra $A$ such that:
\begin{enumerate}
\item[(i)] $\delta(A)\subseteq A$;
\item[(ii)] there exists $s'\in A$ such that $\delta(s')=1$ and $s'$ is a unit in $A$;
\item[(iii)] $A\supseteq B$;
%\item[(iv)] $A_0=A\cap \colinchange{\sout{\mathbb{C}} k}$;
\item[(iv)] each maximal ideal $P$ of $A_0$ has the property that $PA$ is a prime ideal of $A$ and
$A/PA \otimes_{A_0/P} \overline{A_0/P}$ is an integral domain whose field of fractions is a field extension $K_P$ of $\overline{A_0/P}$ of transcendence degree one and genus $g_K$;
\item[(v)] there is some fixed $d'\ge 1$ such that for a Zariski dense set of maximal ideals $P$ of $A_0$, the field $K_P$ is an extension of the $\overline{A_0/P}$-subfield of $F_P$ generated by the image of $s'$ of fixed degree $d'$. 
\end{enumerate}

Since $A$ is noetherian, the set $S$ of monic polynomials in $A[x;\delta]$ is an Ore set \cite[Lemma 1.5.1]{Cohn}.  Then since $A$ contains $B$ and the leading coefficients of the $q_i(x)$ and $s(x)$ are units in $A$ we see that the embedding of $C[t;\mu]$ into $K(x;\delta)$ sends $C[t;\mu]$ into $S^{-1}A[x;\delta]$, since a generating set for $C[t;\mu]$ is sent into this ring.  We also note that prime ideals of $A$ that are closed under application of the derivation $\delta$ lift to prime ideals of $A[x;\delta]$ and, moreover, they survive when we invert $S$ since the elements of $S$ are all regular modulo these prime ideals.

Then for each maximal ideal $P$ of $A_0$, the composition of maps
$$C[t;\mu]\to S^{-1}A[x;\delta]\to \bar{S}^{-1}(A/PA)[x;\delta],$$ where $\bar{S}$ is the monic polynomials in $(A/PA)[x;\delta]$, 
gives a map $\phi_P$ from $C[t;\mu]$ to $ \bar{S}^{-1}(A/PA)[x;\delta]$.  Since 
$\bar{S}^{-1}(A/PA)[x;\delta]$ is a domain, the kernel must be a completely prime ideal of $C[t;\mu]$.  

Moreover, since the embedding is the identity on $k$, we see that the embedding $\iota$ sends $C_0\subseteq C\cap k$ into $A\cap k$ and so $\phi_P$ maps $C_0$ to  $(A\cap k)/(PA\cap k)$.  We claim that $(A\cap k)/(PA\cap  k)$ is an algebraic extension of a finite field.  This will then give that $Q:={\rm ker}(\phi)\cap C_0$ is a maximal ideal of $C_0$ since $\phi_P$ must then map $C_0$ into a finite field since $C_0$ is finitely generated.  To obtain the claim, we note that by construction $A$ is a finitely generated $A_0$-algebra whose Krull dimension is one greater than that of $A_0$.  It follows that $A$ cannot contain a polynomial ring in two variables over $A_0$.  Moreover, $A$ is not algebraic over $k$ and so there exists some $z\in A$ that is transcendental over $k$.  We now claim that if $\alpha\in A\cap k$ then $A_0+A_0\alpha+\cdots$ is not direct; if it were, then since $A$ cannot contain a polynomial ring in two variables over $A_0$, the infinite sum $A_0[\alpha]+A_0[\alpha]z+\cdots $ could not be direct, and this would then give that $z$ is algebraic over $k$.  It follows that every element of $A\cap k$ is algebraic over $A_0$ and hence $(A\cap k)/(PA\cap k)$ is algebraic over the finite field $A_0/P$, thus giving the claim.

By property (4), we have that $QC$ is a prime ideal of $C$ and since $\mu(Q)=0$ we see that this prime ideal is $\mu$-invariant and is in the kernel of $\phi$.  Since $C_0/Q$ and $A_0/P$ are finite fields of the same characteristic, they have isomorphic algebraic closures and thus we get an induced map
$$\bar{\phi}_P:(C/QC\otimes_{C_0/Q} \overline{C_0/Q})[t;\mu] \to (\bar{S}^{-1}(A/PA)\otimes_{A_0/P} \overline{A_0/P})[x;\delta].$$

We claim that $\bar{\phi}_P$ is injective.  We observe that once we have this, we are done, because $\bar{\phi}_P$ will induce an injection from the division ring $F_P(t;\mu)$ into $K_P(x;\delta)$, by localizing.  If we choose a maximal ideal $P$ such that $A_0/P$ has characteristic $p>\max(d,d')$ and such that $F_P$ is degree $d$ over the subfield generated by $s$ and $K_P$ is degree $d'$ over the subfield generated by $s'$ (this is possible since we can invert the set of primes $p\le \max(d,d')$ in $A_0$ and we will still have an infinite spectrum and for a Zariski dense set of $P$ we will get the desired degrees), then by Lemmas \ref{lem: Kp} and \ref{lem:Z}, $F_P(t;\mu)$ and $K_P(x;\delta)$ are both $p^2$-dimensional over their respective centres and so by Lemma \ref{lem:Z2} we have that this embedding restricts to an embedding of their centres and by Lemma \ref{lem:Z}, this then gives an embedding of 
$F_P(y)$ into $K_P(y)$, where $y$ is an indeterminate, which gives that the genus of $F_P$ is at most the genus of $K_P$ by Lemma \ref{lem:genus}.  Since we have that the genera of $F_P$ and $K_P$ are respectively $g_F$ and $g_K$, we see that
$$g_F\le g_K,$$ as desired.

Thus it only remains to show that the map $\bar{\phi}_P$ is injective.  
Let $R=(C/QC\otimes_{C_0/Q} \overline{C_0/Q})$.  Then
$R$ is a finitely generated commutative $\overline{C_0/Q}$-algebra of Krull dimension one.  It follows that $R$ has Gelfand-Kirillov dimension one \cite[Theorem 4.5 (a)]{KL}, and so since $R$ is finitely generated, $R[x;\mu]$ has Gelfand-Kirillov dimension two \cite[Proposition 3.5]{KL}.  
Since $R$ is a domain, we see that if $\bar{\phi}_P$ is not injective then there is some nonzero prime ideal $I$ of $R[t;\mu]$ such that $I$ is equal to the kernel of $\bar{\phi}_P$.  Moreover, $I$ must be a completely prime ideal since $\bar{\phi}$ maps into a domain.  Now the Gelfand-Kirillov dimension of $R[t;\mu]/I$ is at most one as a $\overline{C_0/Q}$-algebra if $I$ is nonzero \cite[Proposition 3.15]{KL}. But this then gives that $R[t;\mu]/I$ is commutative, as it is a domain of Gelfand-Kirillov dimension one over an algebraically closed field.  (This is a now well-known observation that uses Tsen's theorem \cite{Tsen} and the Small-Stafford-Warfield theorem \cite{SSW}.)  Now by (2), we have $[t,s]=\delta(s)$ and $s$ and $\delta(s)$ are units.  But if $I$ is nonzero then $t$ and $s$ commute modulo $I$ and so $[t,s] = \delta(s)\in {\rm ker}(\phi)_P$, a contradiction.  The result follows.\end{proof}
\section{Additional remarks about embeddings}
The general question as to when there exists an embedding of $D_1$ into $D_2$ when $D_1$ and $D_2$ are two division rings on Artin's list has been looked at before and there are many folklore results in this area.  We can divide Artin's list into four types of division rings:
\begin{enumerate}
\item[(1)] those that are finite-dimensional over their centres;
\item[(2)] the Sklyanin division rings not finite over their centres;
\item[(3)] Skew field extensions of automorphism type (not finite over the centre);
\item[(4)] Skew field extensions of derivation type (not finite over the centre).
\end{enumerate}
Theorem \ref{thm: main} then addresses embeddings for division rings of Type 4.  In fact, for embeddings of some of the other types much more is already known.  For example, maximal subfields of division rings of Types 2--4 have transcendence degree one over $k$ (cf. \cite[Theorem 1.4]{Be}) and hence no division ring of Type 1 can embed into one of another type; conversely, division rings of Types 2--4 all contain free algebras on two generators and hence cannot embed into division rings of Type 1.  

In the case of embedding division rings of Type 1, we have the following easy observations.
\begin{prop}
Let $D_1,D_2$ be division algebras that are finite-dimensional over central fields $K,F$ which are finitely generated of transcendence degree 2 over $k$.
Then if there is an embedding of $D_1$ into $D_2$ then the period of $D_2$ is at least that of $D_1$.  If furthermore, their periods are equal, then the plurigenera,  irregularity, geometric genus and Kodaira dimension of $Z(D_2)$ are at least those of $Z(D_1)$.
\end{prop}
\begin{proof}
Since the period of $D_i$ is its PI degree we immediately obtain the first statement.  The last statement follows from Lemmas~\ref{lem:Z2} and~\ref{bndschar0}.
\end{proof}
\begin{quest}
With the above hypotheses, can we conclude that the Kodaira dimension of 
$D_2$ is at least that of $D_1$?
\end{quest}
We note that the division ring of Type 3 given by $ k(t)(x;\sigma)$ where $\sigma(t)=t+1$ is in fact isomorphic to the division ring $k(y)(z;\delta)$ where $\delta$ is differentiation with respect to $y$.  Here the automorphism is given by $t=yz$ and $x=z$.  This division ring embeds into every division ring of Type 4 since, by working up to isomorphism, we may assume that there is always a solution to $\delta(s)=1$ when our field is a finitely generated transcendence degree one field extension of $k$.  Other than these trivial cases, it is known  in some cases that some division rings of Type 4 do not embed into division rings of Type 3 and it is known that some division rings of Type 3 do not embed into division rings of Type 4.  Perhaps the easiest such example is given by taking $k(t)(x;\sigma)$ where $\sigma$ is not conjugate to an automorphism of the form $t\mapsto t+1$.  In this case, one can perform a change of variables and assume that $\sigma(t)=qt$ for some nonzero element of $k$. Moreover, since we are assuming we are not finite over our centre, we have that $q$ is not a root of unity.  The quotient division ring of a ring of the form $K[t;\delta]$ embeds in a skew power series ring $K((t^{-1};\delta))$ and a simple computation by looking at leading terms shows that there are no solutions to the equation $xt=q tx$ in this ring with $q\neq 1$.  In particular, $k(t)(x;\sigma)$ cannot embed into $K((t^{-1};\sigma))$ and hence cannot embed into any division ring of type 3.

Finally, while not explicitly written down in the literature, Artin \cite{Art} points out that a Sklyanin division algebra is a ring of invariants of a division ring of Type 3---specifically there is an elliptic curve $E$ and an infinite-order translation $\sigma$ of $E$ such that a Sklyanin division ring is the ring of invariants of a $\mathbb{Z}/2\mathbb{Z}$-action on $k(E)(t;\sigma)$, where the action on $k(E)(t;\sigma)$ comes from the induced map on $k(E)$ from the negation map on $E$ and then by extending the action by sending $t$ to $t^{-1}$.\footnote{Artin attributes this (non-trivial) observation to Michel Van den Bergh.}  
%\colincomment{I am pretty sure this example is in David Patrick's thesis and maybe elsewhere.  Sue was asking me about it not long ago.  It is the following:
%Let $E$ be an elliptic curve.  We define a infinite dihedral group action on 
%$k(E)$ where $\sigma$ acts by translation by an infinite order point and $\tau$ acts by negation.  Then $D_\infty=\langle \sigma,\tau \st \tau^2=1, \tau \sigma \tau =\sigma^{-1} \rangle$ acts on $k(E)$.  Define an action of $\tau$ on 
%$k(E)(t;\sigma)$ by $\tau(t)=t^{-1}$.  Then the Sklaynin division algebra is
%$$((E)(t;\sigma)){\langle \tau \rangle}.$$
%I don't know where this claim is proved, but maybe Sue knows or maybe David's thesis has a reference.}\jasoncomment{Well, we can see if it is in the thesis.  I'm not sure if his thesis is online.}
This gives an embedding of a division ring of Type 2 into one of Type 3.  We are not aware of any additional results involving embeddings from or into the Sklyanin division rings and a systematic study of the possible embeddings of the division rings on Artin's list would make an interesting topic for future study.   We conclude by asking about a generalization of Theorem \ref{thm: main}, which would give a complete understanding of embeddings between division rings of Type 4 if the question were answered affirmatively.
\begin{quest}
Let $k$ be an algebraically closed field of characteristic zero, 
let $X$ and $Y$ be smooth irreducible algebraic curves over $k$, and let $D(X)$ and $D(Y)$ denote respectively the quotient division rings of the ring of differential operators of $X$ and $Y$. 
 If $d$ is a natural number, is it the case that there is an embedding of $D(X)$ into $D(Y)$ of degree $d$ if and only if there is a degree $d$ surjective morphism from $Y$ to $X$?  
\end{quest}
While it is known that for division rings $D_1$ and $D_2$ from Artin's list with $D_1\subseteq D_2$, we have that $D_2$ is finite-dimensional as a left and right $D_1$-vector space \cite[Theorem 1.4]{Be} (see also Schofield \cite[Corollary 35]{Sch}), it is not known that these two dimensions coincide.  In general, examples where these two quantities are different exist (see Cohn \cite[Section 5.9]{Cohn1}), although there are no known counter-examples for division rings from Artin's list.  We thus define the degree for an embedding of $D(X)$ into $D(Y)$ to be the minimum of the dimensions of $D(Y)$ as a left and right $D(X)$-vector space.
We point out that one direction is trivial, but the other direction would have important implications beyond the inequality between genera given in Theorem \ref{thm: main}.  For example, it would show that the gonality of the curve $X$ bounds the degree of the embedding of the Weyl division algebra into $D(X)$. 
%For the reader's convenience we summarize what is known about embeddings of a division ring $D_1$ into a division ring $D_2$ for $D_1$ and $D_2$ of the types 1--4 described above.
%\begin{center}
%\begin{tabular}{c|cccc}
%\backslashbox{$D_1$}{$D_2$} &  1 & 2 & 3 & 4 \\
%\hline
%1 & {\rm PIdeg}(D_1)\le {\rm PIdeg}(D_2) & N & N   & N \\
%2 & N & Y & etc & etc \\
%3 & N & Y & etc & etc \\
%4 & N & Y & ? & $g \leq$ 
%\end{tabular}
%\end{center} 


\begin{thebibliography}{99}
\bibitem{Art} M. Artin, Some problems on three-dimensional graded domains. \emph{Representation theory and algebraic geometry} (Waltham, MA, 1995), 1--19, \emph{London Math. Soc. Lecture Note} Ser., 238, Cambridge Univ. Press, Cambridge, 1997.
\bibitem{Be} J. P. Bell, 
Division algebras of Gelfand-Kirillov transcendence degree 2. 
\emph{Israel J. Math.} {\bf 171} (2009), 51--60.
\bibitem{BSm} J. P. Bell and A. Smoktunowicz, 
Rings of differential operators on curves. 
\emph{Israel J. Math.} {\bf 192} (2012), no. 1, 297--310. 
\bibitem{CI} D. Chan and C. Ingalls, The minimal model program for orders over surfaces. \emph{Invent. Math.} {\bf 161} (2005), no. 2, 427--452.
\bibitem{Cohn1} P. M. Cohn. \emph{Skew fields: theory of general division rings.} Cambridge Univ. Press, Cambridge, USA, 1995.
\bibitem{Cohn} P. M. Cohn, \emph{Free ideal rings and localization in general rings.} New Mathematical Monographs, 3. Cambridge University Press, Cambridge, 2006.
\bibitem{EGA IV_2} A. Grothendieck and J. Dieudonn\'e,  \'El\'ements de g\'eom\'etrie alg\'ebrique. IV. \'Etude locale des sch\'emas et des morphismes de sch\'emas II. \emph{Inst. Hautes \'Etudes Sci. Publ. Math.} No. 24 1965, 231 pp.
\bibitem{EGA IV_4} A. Grothendieck and J. Dieudonn\'e, \'El\'ements de g\'eom\'etrie alg\'ebrique. IV. \'Etude locale des sch\'emas et des morphismes de sch\'emas IV. \emph{Inst. Hautes \'Etudes Sci. Publ. Math.}  No. 32 1967, 361 pp.
\bibitem{GK}  I. M. Gelfand and A. A. Kirillov, Sur les corps li\' es aux alg\` ebres enveloppantes des
alg\`ebres de Lie. \emph{Inst. Hautes \'Etudes Sci. Publ. Math.}  {\bf 31} (1966), 5--19.
\bibitem{GS} Goodearl, K. R.; Stafford, J. T. The graded version of Goldie's theorem. Algebra and its applications (Athens, OH, 1999), 237--240, \emph{Contemp. Math.}, {\bf 259}, Amer. Math. Soc., Providence, RI, 2000. 
\bibitem{Hart} R. Hartshorne, \emph{Algebraic geometry.} Graduate Texts in Mathematics, No. 52. Springer-Verlag, New York-Heidelberg, 1977.
\bibitem{Jou} J.-P. Jouanolou, \emph{Th\'eor\`emes de Bertini et applications.} Progress in Mathematics, 42. Birkh\"auser Boston, Inc., Boston, MA, 1983.
\bibitem{KL} G. Krause and T. Lenagan, \emph{Growth of Algebras and Gelfand-Kirillov Dimension, revised edition.} Graduate Studies in Mathematics, no. 22. American Mathematical Society, Providence, RI, 2000.
\bibitem{ML} L. Makar-Limanov, The skew field of fractions of the Weyl algebra contains a free noncommutative subalgebra. \emph{Comm. Algebra} {\bf 11} (1983), no. 17, 2003--2006.
%%\bibitem{MP} J. Miyaoka and T. Peternell, \emph{Geometry Of Higher Dimensional Algebraic Varieties}
%% DMV Seminar, 26. Birkh\"auser Verlag, Basel, 1997.
\bibitem{Sch} A. H. Schofield. Stratiform simple artinian rings. \emph{Proc. London Math. Soc.} {\bf 53}(1986), 267--287.
\bibitem{SSW} L. W. Small, J. T. Stafford, and R. B. Warfield, Jr. Affine algebras of Gelfand-Kirillov dimension one are PI.  \emph{Math. Proc. Cambridge Philos. Soc.} {\bf 97} (1985), no. 3, 407--414.
\bibitem{Tsen} C. Tsen, Divisionsalgebren \"uber Funktionenk\"orpern, Nachr. Ges. Wiss. G\"ottingen, Math.-Phys. Kl, (1933) 335--339. 
\bibitem{YZ} A. Yekutieli and J. J. Zhang, Homological transcendence degree. \emph{Proc. London Math. Soc.} (3) {\bf 93} (2006), no. 1, 105--137.
\bibitem{Z} J. J. Zhang,  On Lower Transcendence Degree.  \emph{Adv. Math.} {\bf 139} (1998), 157--193.
\bibitem{Z2} J. J. Zhang, On Gelfand-Kirillov transcendence degree. 
\emph{Trans. Amer. Math. Soc.} {\bf 348} (1996), no. 7, 2867--2899. 
\end{thebibliography}
\end{document}